\documentclass[a4paper,11pt]{article}
 \usepackage[english]{babel}
\usepackage[a4paper]{geometry}
\usepackage{amsmath}
\usepackage{amsthm}
\usepackage{amssymb}
\usepackage{bbm}
\usepackage{color}
\usepackage[dvipsnames]{xcolor}

\usepackage{tikz}
\usetikzlibrary{positioning}
\usetikzlibrary{arrows}
\usepackage{caption}
\usepackage{subcaption}
 
\usepackage{mathrsfs}
\usepackage{enumerate}
 
\usepackage{hyperref} 

\def\bE{\mathbb{E}}

\def\be{\begin{equation}}
\def\ee{\end{equation}}

\DeclareMathOperator{\sech}{sech}

\newtheorem{theorem}{Theorem}   

\newtheorem*{conj*}{Conjecture}
\theoremstyle{remark}
\newtheorem{remark}[theorem]{Remark}

\begin{document}
\title{On the local nature of the de Almeida-Thouless line for mixed \(p\)-spin glasses}
\author{Jean-Christophe Mourrat\thanks{Department of mathematics, ENS Lyon and CNRS, Lyon, France} 
\and Adrien Schertzer\footnotemark[1]}
\maketitle

\begin{abstract}
Jagannath and Tobasco~\cite{JagTob} proposed a generalized de Almeida-Thouless (AT) criterion aimed at characterizing the replica symmetric (RS) regime for a broad class of mixed \(p\)-spin glass models with Ising spins. In this paper, we show that this generalized AT condition does not characterize the RS regime in general. Using the Hopf-Lax representation of the Parisi formula, we construct explicit counterexamples within the class of mixed \(p\)-spin models. In particular, we exhibit a model in which the classical AT perturbation is performed around the unique minimizer of the RS free energy, and prove that even in this setting, the AT criterion fails to characterize the RS phase. By contrast, the validity of the classical AT condition for the Sherrington-Kirkpatrick model remains open.
\end{abstract}


\section{Introduction} 

\subsection{The model}

We study mixed $p$-spin glass models, which we now briefly recall.
Let $(\beta_p)_{p\ge2}$ be a sequence of non-negative coefficients, which for
simplicity we assume to have only finitely many non-zero entries, and define
\[
\xi_0(r)=\sum_{p\ge2} \beta_p\, r^p,
\qquad r\in\mathbb R.
\]
Consider the hypercube $\Sigma_N=\{-1,1\}^N$ equipped with the Hamiltonian
$H_N=(H_N(\sigma))_{\sigma\in\Sigma_N}$, a centered Gaussian process with
covariance
\[
\mathbb{E}\bigl[H_N(\sigma^1)H_N(\sigma^2)\bigr]
=
N\,\xi_0\!\left(\frac{\sigma^1\cdot\sigma^2}{N}\right).
\]
Such a field can be constructed explicitly as a linear combination of independent
Gaussian tensors,
\[
\sum_{1\le i_1,\dots,i_p\le N}
J_{i_1,\dots,i_p}\,\sigma_{i_1}\cdots\sigma_{i_p},
\]
where the coefficients $(J_{i_1,\dots,i_p})$ are i.i.d.\ standard Gaussian random
variables.  The classical Sherrington-Kirkpatrick (SK) model \cite{SK} corresponds to the case with $\xi_0(r)=\tfrac12 r^2$. The associated partition function is defined by
\[
Z_N(\beta, h)
=
\frac{1}{2^N}\sum_{\sigma\in\Sigma_N}
\exp\bigl(\beta H_N(\sigma)+h \sum_{i=1}^N \sigma_i \bigr)=: E_0\exp\bigl(\beta H_N(\sigma)+h \sum_{i=1}^N \sigma_i \bigr),
\]
where $\beta>0$ denotes the inverse temperature, $h \in \mathbb R$ is an external field, and we denote by $E_0$ the expectation with respect to the uniform measure on $\Sigma_N$. A central object of interest is the limit of the free energy
\[
F_N(\beta,h):=\frac1N \bE\log Z_N(\beta, h),
\]
as $N$ tends to infinity, whose identification is given by the celebrated Parisi formula.

\subsection{The Parisi formula}

In the next two sections, we follow the conventions of \cite{JagTob}.
Set $\xi(r)=\beta^2\xi_0(r)$.
The Parisi functional is defined by
\begin{equation}
\label{eq:ParisiFunctional}
\mathcal P(\mu;\xi_0,\beta,h)
=
u_\mu(0,h)
-
\frac12
\int_0^1 \xi''(s)\,\mu([0,s])\,s\,ds,
\end{equation}
where $\mu\in\mathcal P([0,1])$ is a probability measure on the unit interval, and
$u_\mu$ denotes the unique weak solution of the Parisi partial differential equation
\begin{equation}
\label{eq:ParisiPDE}
\begin{cases}
\partial_r u_\mu(r,x)
+
\dfrac{\xi''(r)}{2}\,\partial_{xx}u_\mu(r,x)
+
\mu([0,r])\bigl(\partial_x u_\mu(r,x)\bigr)^2
=0,
& (r,x)\in(0,1)\times\mathbb R,\\[0.2cm]
u_\mu(1,x)=\log\cosh(x).
\end{cases}
\end{equation}
The Parisi formula states that
\begin{equation}
\label{eq:ParisiVariational2}
\lim_{N\to\infty}F_N(\beta,h)
=
\min_{\mu\in\mathcal P([0,1])}
\mathcal P(\mu;\xi_0,\beta,h).
\end{equation}
Rigorous proofs of this formula were obtained in
\cite{Gue,Tal,Tal1,Tal2,P}, following the seminal insights of physicists
\cite{MPV}.

Depending on the parameters $(\beta,h)$, the minimizer $\mu$ may have different
structures. The unique minimizer (see \cite{AC1}) is called the Parisi measure.
If $\mu$ is a Dirac measure, the system is said to be in the
\emph{replica symmetric} (RS) phase.
A fundamental problem is to characterize the region in the $(\beta,h)$-plane
where the Parisi measure is replica symmetric.

\subsection{The generalized AT line}

Let $Z$ be a standard Gaussian random variable and define
\begin{align*}
Q^*(\beta,h)
&=
\Bigl\{q\in[0,1]:
\mathbb E\bigl[\tanh^2(\xi'(q)Z+h)\bigr]=q\Bigr\},\\
\alpha(q,\beta,h)
&=
\xi''(q)\,
\mathbb E\bigl[\sech^4(\xi'(q)Z+h)\bigr],\\
\alpha(\beta,h)
&=
\min_{q\in Q^*(\beta,h)}\alpha(q,\beta,h).
\end{align*}
Following \cite{JagTob}, we call the level set $\{\alpha(\beta,h)=1\}$ the \emph{generalized de Almeida-Thouless (AT) line}.
We define
\begin{equation}
\label{eq:ATdef}
\mathrm{AT}
:=
\bigl\{(\beta,h):\alpha(\beta,h)\le1\bigr\},\\
\end{equation}
\begin{equation}
\label{eq:RSdef}
\mathrm{RS}
:=
\bigl\{(\beta,h): \text{the Parisi measure is a Dirac measure } \delta_q \text{ for some } q\in[0,1]\bigr\}.
\end{equation}
Jagannath and Tobasco conjectured that these two regions coincide
\cite[Conjecture~1.6]{JagTob}.

It was shown in \cite[Theorem~1.8]{JagTob} and in \cite{Ch0} that, for any model $\xi_0$,
one always has the inclusion $\mathrm{RS}\subset\mathrm{AT}$. We show here that the converse inclusion is false in general.

\begin{theorem}\label{theorem:main}
There exist mixed \(p\)-spin models \(\xi_0\) with \(\xi_0''(0)>0\) such that
\[
\mathrm{AT}\not\subset \mathrm{RS}.
\]
\end{theorem}
\begin{remark} \label{rem:singl}
It is natural to suspect that the issue with the criterion from~\cite{JagTob} lies in the choice of the ``correct'' overlap parameter \(q\); more precisely, instead of taking a minimum over all $q \in Q^*(\beta,h)$ in the definition of $\alpha(\beta,h)$, one should evaluate $\alpha(\cdot, \beta, h)$ at a minimizer of the RS functional
\begin{equation}  
\label{e.RS}
\left\{
\begin{array}{rcl}  
[0,1] & \to & \mathbb R \\
q & \mapsto & \mathcal P(\delta_q; \xi_0, \beta, h). 
\end{array}
\right.
\end{equation}
A minimizer of this function is automatically an element of $Q^*(\beta,h)$, so by using the set of minimizers of this function in place of $Q^*(\beta,h)$ in the definition of $\alpha(\beta,h)$, we make the set $\mathrm{AT}$ smaller; one can ask if that smaller set is indeed a subset of $\mathrm{RS}$. 

We will show that this is also false. Precisely, we will construct examples of models $\xi_0$ with $\xi_0''(0) > 0$ and choices of $\beta$ and $h$ for which the function in \eqref{e.RS} has a unique minimizer $q(\beta,h)$ and we have $\alpha(q(\beta,h),\beta,h) < 1$, and yet the corresponding Parisi measure is not a Dirac mass.
\end{remark}

\begin{remark}
\label{r.pure.psin}
The assumption \(\xi_0''(0)>0\) is natural and excludes a trivial degeneracy of the AT condition. Indeed, if \(\xi_0''(0)=0\) and \(h=0\), then \(\alpha(\beta,0)=0\) for all \(\beta \geq 0\). In particular,
\[
\mathrm{AT} \cap \{h=0\} = \mathbb{R}\times\{0\},
\]
which cannot coincide with the RS region (see, e.g., ~\cite[Theorem~3]{AC}). 
\end{remark}





\subsection{The AT criterion}

In the specific case of the SK model, the de Almeida-Thouless condition was originally
derived by de Almeida and Thouless~\cite{AT} as a criterion for the instability of
the replica symmetric solution.
One of the first rigorous mathematical results in this direction was obtained by
Toninelli~\cite{Ton}, who proved the inclusion $\mathrm{RS}\subset\mathrm{AT}$.
The Parisi functional admits a first-order optimality condition of the following
form: a probability measure $\mu$ on $[0,1]$ minimizes the Parisi functional if and
only if it cannot be improved by any \emph{mixing variation}, that is, by replacing
an infinitesimal fraction of $\mu$ by a Dirac mass at an arbitrary point
$q\in[0,1]$. Equivalently, for every $q$, the directional derivative of the Parisi
functional in the direction $\delta_q-\mu$ must be non-negative. This is made
possible by the convexity of the Parisi functional (see \cite{JagTob, Ch0} for more details).

In the replica symmetric case $\mu=\delta_{q^*}$, this condition has a simple
interpretation: the parameter $q^*$ must be a global minimizer of a certain
one-dimensional function $G_{\mu}(t)$ naturally associated with $\mu$ (see, for
instance, Proposition~1.1 in \cite{JagTob}). The classical de Almeida-Thouless
condition is recovered by studying this function locally at $t=q^*$: the vanishing
of its first derivative yields the RS self-consistency equation (the
equation characterizing $Q^*(\beta,h)$), while the sign of its second derivative
produces the classical AT stability condition for the SK model \cite{AT}, namely
$\alpha(q^*,\beta,h)\le 1$. 


We also mention that Chen~\cite{Ch}  proved that the de Almeida-Thouless condition does characterize the replica symmetric region for the SK model for $h = 0$ (also allowing for the presence of an additional centered Gaussian external field), a result that was recently extended to spins-glass models with multiple species in \cite{K}. For the SK model with $h = 0$, Zhou~\cite{Z1}  described the precise transition to a full RSB regime just below the critical temperature predicted by the AT criterion. For other pure $p$-spin models, i.e.\ in the case $\xi(r) = r^p$ for an integer $p \ge 3$, and for $h = 0$, the AT criterion becomes vacuous, as discussed in Remark~\ref{r.pure.psin}. In this case, Zhou~\cite{Z2} identified the transition temperature and described the transition from the RS to the 1RSB regime. 

A failure of a local optimality condition in the spirit of the AT criterion was also observed in~\cite[Section 2.1]{P1} for a spin-glass model with spins taking values in $\{-1,0,1\}$.

In~\cite{Pl}, Plefka suggested that two distinct stability conditions may arise in
the SK model, and an interpretation of Plefka's conditions was later proposed
in~\cite{GSS}.

Finally, in~\cite{Bolt2, BY}, the authors investigated the de Almeida-Thouless
condition starting from the finite-$N$ system with algorithms, rather than from the Parisi
variational formula, thereby avoiding the heavy Parisi machinery.

\section{Proof of Theorem~\ref{theorem:main}}

\begin{proof}

\emph{\textbf{Step~1. Hopf-Lax representation of the Parisi formula.}}\\

For every $\beta,l\ge0$, we introduce the \emph{enriched} free energy
\begin{equation}
\label{eq:enriched}
F_N^{\circ}(\beta,l)
:=
-\frac{1}{N}\,
\mathbb{E}\log 
E_0 
\exp\!\Bigl(
\beta\, H_N(\sigma)
-\frac{N\beta^2}{2}\,\xi_0(1)
+ \sqrt{2l}\, z\cdot\sigma
- l\, N
\Bigr),
\end{equation}
where $z=(z_1,\dots,z_N)$ is a vector of independent standard Gaussian random
variables, independent of $H_N$.
With this normalization, the annealed free energy vanishes identically. In
particular,
\[
F_N^{\circ}(0,l)
=
F_1^{\circ}(0,l)
=
-\,\mathbb{E}\log 
\cosh\!\bigl(\sqrt{2l}\, z_1\bigr)
+ l.
\]

Let $\xi^\ast$ denote the convex conjugate of $\xi$,
\[
\xi^\ast(a)
=
\sup_{r\ge0}\bigl(ar-\xi(r)\bigr),
\qquad a\ge0.
\]
Setting $t=\beta^2/2$, Proposition~4.1 of~\cite{MourratParisi} yields, for every
$t>0$,
\begin{align}
\label{eq:HopfLax}
\lim_{N\to\infty}
\frac1N\,\mathbb E \log 
E_0 \exp\!\bigl( \sqrt{2t}\, H_N(\sigma) \bigr)
\le  
t\,\xi(1)
-
\sup_{l \ge 0}
\Bigl\{
F_N^{\circ}(0,l)
-
t\,\xi^\ast\!\left(t^{-1}l\right)
\Bigr\}.
\end{align}
Consequently, 
\begin{align*}
\lim_{N\to\infty}
\frac1N\,\mathbb E \log 
E_0 \exp\!\bigl( \sqrt{2t}\, H_N(\sigma) \bigr)
&\leq
t\,\xi(1)
+
\mathbb{E}\log \cosh\!\bigl(\sqrt{2l}\, z_1\bigr)
- l
+
t\,\xi^\ast\!\left(t^{-1}l\right)\\
&=
t\,\xi(1)
- l^2 + o(l^2)
+
t\,\xi^\ast\!\left(t^{-1}l\right),
\end{align*}
where the last line follows from the Taylor expansion as $l\to0$.

\vspace{0.3cm}

\emph{\textbf{Step~2. Choice of the model.}}

\vspace{0.3cm}

We now specialize to the case of zero external field, \(h=0\), and we claim that
\[
\mathrm{RS}\cap(\mathbb{R}_+\times\{0\})
=
\bigl\{(\beta,0): \lim_{N\to\infty} F_N(\beta,0)
=\tfrac{\beta^2}{2}\,\xi_0(1) \bigr\}.
\]
This follows from the fact that, by definition,
\[
\mathrm{RS}
:=
\bigl\{(\beta,h): \text{the Parisi measure is a Dirac measure } \delta_q \text{ for some } q\in[0,1]\bigr\},
\]
and from the observation that, when \(h=0\), the point \(0\) always belongs to the
support of the Parisi measure (see~\cite[Theorem~1]{AC}). Consequently, in the
RS regime at zero field, the Parisi measure must be \(\delta_0\) (by uniqueness
of the Parisi minimizer).  In this case, a direct inspection of the Parisi
formula shows that
\[
\lim_{N\to\infty} F_N(\beta,0) = \frac{\beta^2}{2}\,\xi_0(1),
\]
and the claim follows. We now consider the mixed model
\[
\xi_0(r)=\frac12 r^2+\frac{1}{p}(Cr)^p,
\]
where $p\ge3$ is fixed and $C>0$ will be chosen large. This corresponds to a
superposition of the SK model with a $p$-spin interaction at low temperature.
Due to the local nature of the AT condition, this additional interaction is not
detected. For this choice of $\xi_0$, we estimate
\begin{align*}
\xi_0^\ast\!\left(\frac{l}{t}\right)
&=
\sup_{r\ge0}
\Bigl(\frac{l}{t}r-\frac12 r^2-\frac{1}{p}(Cr)^p\Bigr)
\le
\sup_{r\ge0}
\Bigl(\frac{l}{t}r-\frac{1}{p}(Cr)^p\Bigr)
=
\frac{p-1}{p}
\Bigl(\frac{l}{Ct}\Bigr)^{\frac{p}{p-1}}.
\end{align*}
Choosing $l=C^{-1}$, we obtain
\begin{align*}
\lim_{N\to\infty}
\frac1N\,\mathbb E \log 
E_0 \exp\!\bigl( \sqrt{2t}\, H_N(\sigma) \bigr)
\le
t\,\xi_0(1)
-
\frac{1}{C^2}
+
o\!\left(\frac{1}{C^2}\right)
+
t\,\frac{p-1}{p}
\Bigl(\frac{1}{C^2 t}\Bigr)^{\frac{p}{p-1}}.
\end{align*}
For $C=C(\beta,p)$ sufficiently large, this yields
\[
\lim_{N\to\infty}
\frac1N\,\mathbb E \log 
E_0 \exp\!\bigl( \beta\, H_N(\sigma) \bigr)
<
\frac{\beta^2}{2}\,\xi_0(1).
\]
Thus, for any fixed $\beta$, one can choose $C(\beta,p)$ large enough so that the
resulting model is not replica symmetric.

We now examine what the generalized AT criterion predicts in this situation.
For the present model, one has $0\in Q^*(\beta,0)$, and therefore
\[
\alpha(\beta,0)
=
\min_{q\in Q^*(\beta,0)} \alpha(q,\beta,0)
\le \alpha(0,\beta,0)
= \beta^2.
\]
In particular, for every $\beta<1$, we have $(\beta,0)\in \mathrm{AT}$,
independently of the value of $C$. However, for every $\beta<1$, one may choose
$C(\beta,p)$ large enough so that the system is no longer replica symmetric.
This shows that the criterion of~\cite{JagTob} fails to characterize the $\mathrm{RS}$ region. It remains to prove what we claim in Remark \ref{rem:singl}. We define the critical
temperature
\[
\beta_c:=\sup\{\beta\ge0:\ (\beta,0) \in  \mathrm{RS}\}.
\]
As observed above, this is equivalent to
\[
\beta_c:=\sup\{\beta\ge0:\ \lim_{N\to\infty} F_N(\beta,0)
=\tfrac{\beta^2}{2}\,\xi_0(1)\}.
\]
By taking $C$ large enough, we have that $\beta_c<1$. By continuity of $\beta\mapsto \lim_{N\to\infty} F_N(\beta,0)$ (since each $F_N(\beta,0)$ is convex in $\beta$ and $F_N(\beta,0)\to F(\beta)$ pointwise with a finite limit, the limiting free energy $F$ is also convex in $\beta$), we have
\[
\lim_{N\to\infty}  F_N(\beta_c,0)
=\tfrac{\beta_c^2}{2} \xi_0(1).
\]
Recall that for every $\beta$, whenever the Parisi minimizer is a Dirac mass, the limiting free energy in the RS regime is given by
\[
\inf_{q\in[0,1]} F_{\mathrm{RS}}(q, \beta),
\]
where
\[
F_{\mathrm{RS}}(q, \beta)
=
\mathbb{E}\!\left[\log\cosh\!\big(\beta \sqrt{\xi_0'(q)}\,Z\big)\right]
+\frac{\beta^2}{2}\Big(\xi_0(1)-\xi_0(q)-(1-q)\xi_0'(q)\Big),
\]
with \(Z\sim\mathcal N(0,1)\). This follows directly from the Parisi formula \eqref{eq:ParisiVariational2} when the minimizer is a Dirac measure.
By uniqueness of the Parisi minimizer, we deduce that $0$ is the unique minimizer
of $F_{\mathrm{RS}}(\cdot, \beta_c)$. We now prove that $0$ is still the unique minimizer of $F_{\mathrm{RS}}(\cdot,\beta_c+\delta)$ for some $\delta > 0$ small enough.  At \(\beta=\beta_c\), the RS functional \(F_{\mathrm{RS}}(\cdot,\beta)\) admits a
\emph{non-degenerate local minimum} at \(q=0\), that is,
\[
\partial_q F_{\mathrm{RS}}(0,\beta_c)=0
\qquad\text{and}\qquad
\partial_q^2 F_{\mathrm{RS}}(0,\beta_c)=\frac{\beta_c^2}{2}(1-\beta_c^2)>0.
\]
Thus, there exist $\varepsilon\in(0,1)$ and $c>0$ such
that
\[
F_{\mathrm{RS}}(q,\beta_c)\ge F_{\mathrm{RS}}(0,\beta_c)+c q^2,
\qquad q\in[0,\varepsilon].
\]
Next, consider the compact set $[\varepsilon,1]$. By continuity of
$(q, \beta)\mapsto F_{\mathrm{RS}}(q,\beta)$, the function
\[
q\longmapsto F_{\mathrm{RS}}(q,\beta_c)-F_{\mathrm{RS}}(0,\beta_c)
\]
attains its minimum on $[\varepsilon,1]$, and by uniqueness of the minimizer at
$\beta_c$ this minimum is strictly positive. Denote
\[
\eta:=\min_{q\in[\varepsilon,1]}
\bigl(F_{\mathrm{RS}}(q,\beta_c)-F_{\mathrm{RS}}(0,\beta_c)\bigr)>0.
\]
By uniform continuity on the compact set $[\varepsilon,1] \times [\beta_c,\beta_c+1]$, there exists
$\delta_0>0$ such that for all $\beta\in[\beta_c,\beta_c+\delta_0]$ and all
$q\in[\varepsilon,1]$, 
\begin{align*}
F_{\mathrm{RS}}(q,\beta)&=F_{\mathrm{RS}}(q,\beta_c)+F_{\mathrm{RS}}(q,\beta)-F_{\mathrm{RS}}(q,\beta_c)\\
&\ge F_{\mathrm{RS}}(q,\beta_c)-\frac{\eta}{2}\\
&=  F_{\mathrm{RS}}(q,\beta_c)-F_{\mathrm{RS}}(0,\beta_c)+ F_{\mathrm{RS}}(0,\beta_c)-\frac{\eta}{2},\\
&\geq F_{\mathrm{RS}}(0,\beta_c)+\frac{\eta}{2},
\end{align*}
the last line by definition of $\eta$. Combining the continuity of $\beta\mapsto F_{\mathrm{RS}}(0,\beta)$ with the latter equation, we obtain
\begin{align*}
F_{\mathrm{RS}}(q,\beta) \geq F_{\mathrm{RS}}(0,\beta)+\frac{\eta}{4}.
\end{align*}
On the other hand, the local quadratic bound around $0$ persists by uniform continuity,
so that, for all $\epsilon_1 > 0$,
$\beta\in[\beta_c,\beta_c+\delta(\epsilon_1)]$ and all $q\in[0,\varepsilon]$, 
\begin{align*}
F_{\mathrm{RS}}(q,\beta)&=F_{\mathrm{RS}}(0,\beta)+F_{\mathrm{RS}}(q,\beta)-F_{\mathrm{RS}}(0,\beta)\\
&\ge F_{\mathrm{RS}}(0,\beta)+F_{\mathrm{RS}}(q,\beta_c)-F_{\mathrm{RS}}(0,\beta_c) -2\epsilon_1\\
& \geq F_{\mathrm{RS}}(0,\beta)+cq^2 -2\epsilon_1.
\end{align*}
Choosing $\delta = \min\bigl(\delta_0,\ \delta(\varepsilon_1)\bigr)$ and combining the two estimates, we conclude that $q=0$ remains the unique global
minimizer of $q\mapsto F_{\mathrm{RS}}(q,\beta)$ for all
$\beta\in[\beta_c,\beta_c+\delta]$.
This concludes the proof and the claim of Remark \ref{rem:singl}.
\end{proof}

\medskip

\noindent \textbf{Acknowledgements.} We gratefully acknowledge the support of the ERC MSCA grant SLOHD (101203974).

\bigskip


\end{document}